\documentclass[12pt]{amsart}
\usepackage{amsmath}
\usepackage{amssymb}
\usepackage{amsfonts,color}
\usepackage{amsthm}
\usepackage{enumerate}
\usepackage[mathscr]{eucal}
\usepackage{verbatim}
\usepackage{amsthm}
\usepackage{amscd}
\usepackage{appendix}

\newtheorem{theorem}{Theorem}[section]

\newtheorem{proposition}[theorem]{Proposition}

\newtheorem{lemma}[theorem]{Lemma}

\theoremstyle{definition}

\newtheorem{remark}{Remark}

\newtheorem{innercustomgeneric}{\customgenericname}
\providecommand{\customgenericname}{}

\newcommand{\newcustomtheorem}[2]{\newenvironment{#1}[1]
  {\renewcommand\customgenericname{#2}
   \renewcommand\theinnercustomgeneric{##1}\innercustomgeneric}{\endinnercustomgeneric}}

\newcustomtheorem{customthm}{Theorem}

\newcommand{\newcustomlemma}[2]{\newenvironment{#1}[1]
  {\renewcommand\customgenericname{#2}
   \renewcommand\theinnercustomgeneric{##1} \innercustomgeneric}{\endinnercustomgeneric}}

\newcustomlemma{customlemma}{Lemma}

\newcustomlemma{customproposition}{Proposition}

\newcommand{\ga}{\gamma}
\newcommand{\Ga}{\Gamma}

\newcommand{\rr}{\mathbb{R}}
\newcommand{\rn}{\mathbb{R}^n}
\newcommand{\zz}{\mathbb{Z}}
\newcommand{\zn}{\mathbb{Z}^n}
\newcommand{\cc}{\mathbb{C}}

\newcommand{\wh}{\widehat}

\numberwithin{equation}{section}

\renewcommand{\le}{\leqslant}
\renewcommand{\ge}{\geqslant}

\def\|{{\boldsymbol{|}}}

\begin{document}

\begin{thanks}
{}
\end{thanks}

\author{Loukas Grafakos}
\address{L. Grafakos, Department of Mathematics, University of Missouri, Columbia, MO 65211, USA} 
\email{grafakosl@missouri.edu}

\author{Bae Jun Park}
\address{B. Park, School of Mathematics, Korea Institute for Advanced Study, Seoul 02455, Republic of Korea}
\email{qkrqowns@kias.re.kr}

\thanks{The first author would like to acknowledge the support of  the Simons Foundation grant 624733. The second author is supported in part by NRF grant 2019R1F1A1044075 and by a KIAS Individual Grant MG070001 at the Korea Institute for Advanced Study} 

\title[Sharp Hardy space estimates for multipliers]{Sharp Hardy space estimates for multipliers}

\begin{abstract} 
We provide an improvement of Calder\'on and Torchinsky's version \cite{Ca_To}  of the  H\"ormander multiplier theorem  on  Hardy spaces $H^p$ ($0<p<\infty$),  substituting  the Sobolev space $L_s^2(A_0)$ by the Lorentz-Sobolev space $L_s^{\tau^{(s,p)} ,\min(1,p) }(A_0)$, where $\tau^{(s,p)}  =\frac{n}{s-(n/\min{(1,p)}-n)}$ and 
$A_0$ is the annulus $\{\xi \in \rn:\,\, 1/2<|\xi|<2\}$. Our theorem also extends   that of Grafakos and Slav\'ikov\'a \cite{Gr_Sl} to the range $0<p \le 1$. Our result is sharp in the sense that the preceding Lorentz-Sobolev space   cannot be replaced by a larger Lorentz-Sobolev space $L^{r,q}_s(A_0)$ with $r< \tau^{(s,p)} $ or $q>\min(1,p)$.
\end{abstract}

\maketitle

\section{Introduction}\label{introduction}

Let $\mathscr S  (\rn)$ denote the Schwartz space and $\mathscr S'(\rn)$ the space of tempered distributions 
on $\rn$. For the Fourier transform of $f\in \mathscr S(\rn)$ we use the definition  $\widehat{f}(\xi):=\int_{\rn}{f(x)e^{-2\pi i\langle x,\xi\rangle}}dx$ and denote by $f^{\vee}(\xi):=\widehat{f}(-\xi)$ the inverse Fourier transform of $f$. We also extend these transforms to the space of tempered distributions.

Given a bounded function $\sigma$ on $\rn$, the multiplier operator $T_{\sigma}$ is defined as
\begin{equation*}
T_{\sigma}f(x):=\int_{\rn}{\sigma(\xi)\widehat{f}(\xi)e^{2\pi i\langle x,\xi\rangle}}d\xi
\end{equation*} for $f\in \mathscr S(\rn)$, 
where $\langle x,\xi\rangle$ is the dot product of $x$ and $\xi$  in $\rn$.
The classical Mikhlin multiplier theorem \cite{Mik} states that if a function $\sigma$, defined on $\rn$, satisfies
\begin{eqnarray*}
\big|  \partial_{\xi}^{\alpha}\sigma(\xi)  \big|\lesssim_{\alpha}|\xi|^{-|\alpha|}, \qquad |\alpha|\le \big[n/2\big]+1,
\end{eqnarray*}  then the operator $T_{\sigma}$ admits a bounded extension in $L^p(\rn)$ for $1<p<\infty$.
In \cite{Ho} H\"ormander sharpened  Mikhlin's result, using the weaker condition
\begin{eqnarray}\label{hocondition}
\sup_{j\in\zz}{\big\Vert \sigma(2^j\cdot)\widehat{\Psi}\big\Vert_{L^2_s(A_0)}}<\infty
\end{eqnarray} for $s>n/2$, where $L^2_s $ denotes  the standard $L^2$-based Sobolev space on $\rn$, $\Psi$ is a Schwartz function on $\rn$ whose Fourier transform is supported in the annulus $A_0=\{\xi:\,\,1/2<|\xi|<2\}$ and satisfies $\sum_{j\in \zz}{\widehat{\Psi}(2^{-j}\xi)}=1$, $\xi\not= 0$. 
 Calder\'on and Torchinsky \cite{Ca_To} proved that if (\ref{hocondition}) holds for $s>n/p-n/2$, then $\sigma$ is a Fourier multiplier of Hardy space $H^p(\rn)$ for $0<p\le 1$. A different proof was given by Taibleson and Weiss \cite{Ta_We}. 
It turns out that the condition $s>n/\min{(1,p)}-n/2$ is optimal for   boundedness to hold and it is natural to ask whether (\ref{hocondition}) can be weakened. 
Baernstein and Sawyer \cite{Ba_Sa} obtained endpoint $H^p(\rn)$ estimates by using Herz space conditions for $\big( \sigma(2^j\cdot)\widehat{\Psi} \big)^{\vee}$. An endpoint $H^1-L^{1,2}$ estimate involving Besov space was given by Seeger \cite{Se1, Se2} and  these estimates were improved and extended to Triebel-Lizorkin spaces by Seeger \cite{Se3} and Park \cite{Park}.
Using an interpolation method, Calder\'on and Torchinsky \cite{Ca_To} 
obtained $L^p$-boundedness  for $T_\sigma$ in the range $|1/p-1/2|<s/n$ ($1<p<\infty$),   
replacing $L_s^2(A_0)$ in (\ref{hocondition}) by the $L^r$-based Sobolev space $L_s^r(A_0)$ for some  $r>n/s$; this   was revisited by Grafakos, He, Honz\'ik, and Nguyen \cite{Gr_He_Ho_Ng} who provided counterexamples 
indicating the optimality of the range of $p$'s. 
 Recently, Grafakos and Slav\'ikov\'a \cite{Gr_Sl} improved this result, replacing (\ref{hocondition}) by 
\begin{equation*}
\sup_{j\in\zz}{\big\Vert \sigma(2^j\cdot)\widehat{\Psi}\big\Vert_{L^{n/s,1}_s(A_0)}}<\infty
\end{equation*}
where $L_s^{n/s,1} $ is a Lorentz-type Sobolev space (defined in \eqref{LorentzSobolev}).  This 
formulation  eliminated the need for the index $r$.

Before stating our  results, we   recall the definition of Lorentz spaces $L^{p,q}(\rn)$ and Lorentz-Sobolev spaces $L^{p,q}_s(\rn)$.
For any measurable function $f$ defined on $\rn$, the decreasing rearrangement of $f$ is defined by
\begin{equation*}
f^*(t):=\inf\big\{s>0: d_f(s)  \le t \big\}, \qquad  t>0
\end{equation*} where $d_f(s):=\big| \{x\in\rn:|f(x)|>s\}\big|$. Here we adopt the convention that the infimum of the empty set is $\infty$.
Then for $0<p,q\le \infty$ we define
\begin{equation*}
\Vert f\Vert_{L^{p,q}(\rn)}:=\begin{cases}
\displaystyle \Big(\int_0^{\infty}{\big(t^{1/p}f^*(t) \big)^{q}}\frac{dt}{t} \Big)^{1/q}, & q<\infty\\
\qquad \displaystyle\sup_{t>0}{t^{1/p}f^*(t)}, & q=\infty.
\end{cases}
\end{equation*} The set of all $f$ with $\Vert f\Vert_{L^{p,q}(\rn)}<\infty$ is called the Lorentz space 
and is denoted by $L^{p,q}(\rn)$.
For $s>0$ let $(I-\Delta)^{s/2}$ be the inhomogeneous fractional Laplacian operator, defined by
\begin{equation*}
(I-\Delta)^{s/2}f:=\big( (1+4\pi^2|\cdot|^2)^{s/2}\widehat{f}\big)^{\vee}.
\end{equation*} 
Then for $0<p,q\le \infty$ and $s>0$ let
\begin{equation}\label{LorentzSobolev}
\Vert f\Vert_{L^{p,q}_s(\rn)}:=\big\Vert (I-\Delta)^{s/2}f\big\Vert_{L^{p,q}(\rn)}.
\end{equation}

\begin{customthm}{A}\cite{Gr_Sl}\label{knownresult}
Let $1<p<\infty$ and $0<s<n$ satisfy
\begin{equation}\label{conditions}
s>\big| n/p-n/2\big|.
\end{equation}
Then there exists $C>0$ such that
\begin{equation*}
\Vert T_{\sigma}f\Vert_{L^p(\rn)}\le C \sup_{j\in \zz}{\big\Vert \sigma(2^j\cdot)\widehat{\Psi}\big\Vert_{L_s^{n/s,1}(\rn)}}\Vert f\Vert_{L^p(\rn)}.
\end{equation*}
\end{customthm}
Moreover, a counterexample showing   that condition (\ref{conditions}) is optimal  can be found in Slav\'ikov\'a \cite{Sl}; this means that $L^p$ boundedness could fail on the line $\big| n/p-n/2\big|=s$.

The purpose of this paper is to extend Theorem A to   Hardy spaces $H^p(\rn)$ for $0<p<\infty$.
Let $\Phi$ be a Schwartz function satisfying $\int_{\rn}{\Phi(x)}dx=1$ and $\textup{Supp}(\widehat{\Phi})\subset \{\xi\in\rn : |\xi|\le 2\}$, and  $\Phi_k:=2^{kn}\Phi(2^k\cdot)$. 
We define $H^p(\rn)$ to be the collection of all tempered distributions $f$ satisfying
\begin{equation*}
\Vert f\Vert_{H^p(\rn)}:=\big\Vert \sup_{k\in\zz}{| \Phi_k\ast f|}\big\Vert_{L^p(\rn)}<\infty.
\end{equation*}
Throughout this paper we fix the index
\begin{equation*}
\tau^{(s,p)}:=\frac{n}{s-(n/\min{(1,p)}-n)}.
\end{equation*}
The first main results of this paper is the following:
\begin{theorem}\label{mainpositive}
Let $0<p<\infty$ and $0<s<n/\min{(1,p)}$ satisfy $(\ref{conditions})$.
Then there exists $C>0$ such that
\begin{equation}\label{mainresultest}
\Vert T_{\sigma}f\Vert_{H^p(\rn)}\le C\sup_{j\in\zz}{\big\Vert  \sigma(2^j\cdot)\widehat{\Psi} \big\Vert_{L_{s}^{\tau^{(s,p)},\min{(1,p)}}(\rn)}}\Vert f\Vert_{H^p(\rn)}.
\end{equation}
\end{theorem}

The above theorem coincides with Theorem \ref{knownresult} if $1<p<\infty$ because $H^p(\rn)=L^p(\rn)$ for $1<p<\infty$, and so we   mainly deal with the case $0<p\le 1$ in the paper. However, a complex interpolation argument between $H^1$- and $L^2$-boundedness yields the result for $1<p<2$; this recovers Theorem \ref{knownresult} by a duality argument, as our proof for $0<p\le 1$ is in fact independent of that of Theorem \ref{knownresult}. 
We provide a sketch of this in the appendix. Actually the construction of analytic family of operators and  interpolation techniques are very similar to those used in \cite{Gr_Sl}. 

\begin{remark}
As a result of Baernstein and Sawyer \cite[Corollary 1 (Chapter 3)]{Ba_Sa}, for $0<p<1$ and $s\ge n/p-n/2$ we have
\begin{equation}\label{basaresult}
\Vert T_{\sigma}f\Vert_{H^p(\rn)}\lesssim \sup_{j\in\zz}{\big\Vert \sigma(2^j\cdot)\wh{\Psi}\big\Vert_{B_{\tau^{(s,p)}}^{s,p}(\rn)}}\Vert f\Vert_{H^p(\rn)}
\end{equation} where $\Psi_k:=2^{kn}\Psi(2^k\cdot)$ and $B_p^{s,q}(\rn)$ is the Besov space with (quasi-)norms
\begin{equation*}
\Vert g\Vert_{B_p^{s,q}(\rn)}:=\Vert \Phi \ast g\Vert_{L^p(\rn)}+\Big(\sum_{k=1}^{\infty}{2^{skq}\big\Vert \Psi_k\ast g\big\Vert_{L^p(\rn)}^q} \Big)^{1/q}.
\end{equation*}  
Then the case $0<p<1$ in (\ref{mainresultest}) could be also obtained as a consequence of (\ref{basaresult}) and the embedding 
\begin{equation}\label{setrresult}
B_{\tau^{(s_0,p)}}^{s_0,p}(\rn)\hookrightarrow L_{s_1}^{\tau^{s_1,p},p}(\rn)\hookrightarrow B_{\tau^{(s_2,p)}}^{s_2,p}(\rn),\quad s_2<s_1<s_0~ \text{ and }~ \tau^{(s_1,p)}>1,
\end{equation} which follows from the recent generalization of the Franke-Jawerth embedding theorem for Triebel-Lizorkin-Lorentz spaces of Seeger and Trebels \cite{Se_Tr}.
Conversely, our result also implies (\ref{basaresult}) for $s>n/p-n/2$ via the embedding (\ref{setrresult}) as Theorem \ref{mainpositive} will be proved in a different way, based on the Littlewood-Paley theory for Hardy spaces and some inequalities in Lorentz spaces. 
We note that when $s=n/p-n/2$, (\ref{basaresult}) holds while (\ref{mainresultest}) fails as mentioned below.

On the other hand, a certain weight condition is required in \cite{Ba_Sa} when we extend (\ref{basaresult}) to $H^1$-boundedness. To be specific, we have
\begin{equation}\label{h1boundresult}
\Vert T_{\sigma}f\Vert_{H^1(\rn)}\lesssim \sup_{j\in\zz}{\big\Vert \sigma(2^j\cdot)\wh{\Psi}\big\Vert_{B_{n/s}^{s,1}(\omega)}}\Vert f\Vert_{H^1(\rn)}, \quad s\ge n/2
\end{equation} where $\{\omega(k)^{-1}\}_{k\in\mathbb{N}}\in \ell^2$ and
\begin{equation*}
\Vert g\Vert_{B_{n/s}^{s,1}(\omega)}:=\Vert \Phi \ast g\Vert_{L^{n/s}(\rn)}+\sum_{k=1}^{\infty}{\omega(k)2^{sk}\big\Vert \Psi_k\ast g\big\Vert_{L^{n/s}(\rn)}}.
\end{equation*}
However, a sharp endpoint $H^1$- boundedness holds   using Lorentz-Sobolev conditions without weights in Theorem \ref{mainpositive}.
This, combined with the embedding (\ref{setrresult}), improves (\ref{h1boundresult}) by replacing $B_{n/s}^{s,1}(\omega)$ by $B_{n/s}^{s,1}$ for $s>n/2$. 
When $s=n/2$, the optimality of $\{\omega(k)^{-1}\}_{k\in\mathbb{N}}\in \ell^2$ for (\ref{h1boundresult}) remains open, but it is known in Park \cite[Theorem 3.4]{Park} that $B_{2}^{n/2,1}(\omega)$ in (\ref{h1boundresult}) cannot be replaced by $B_{2}^{n/2,1}$.

\end{remark}

We now turn our attention to the sharpness of Theorem~\ref{mainpositive}. 
We point out that the example of Slav\'ikov\'a \cite{Sl} is still applicable to the case $0<p\le 1$ with the dilation property $\Vert f(\epsilon \cdot)\Vert_{H^p(\rn)}=\epsilon^{-n/p}\Vert f\Vert_{H^p(\rn)}$, and therefore (\ref{conditions}) is sharp in Theorem \ref{mainpositive}.
We now consider the optimality of different parameters. 
Note that for $0<r_1< r_2<\infty$ and $0<q_1,q_2\le \infty$
\begin{equation}\label{lorentzembedding}
\big\Vert \sigma(2^j\cdot)\widehat{\Psi}\big\Vert_{L_s^{r_1,q_1}(\rn)}\lesssim \big\Vert \sigma(2^j\cdot)\widehat{\Psi}\big\Vert_{L_{s}^{r_2,q_2}(\rn)} \quad \text{~uniformly in }~ j,
\end{equation}
which follows from the H\"older inequality with even integers $s$, complex interpolation technique, and a proper embedding theorem. 
Moreover, if $q_1\ge q_2$, then the embedding $L_s^{r,q_2}(\rn) \hookrightarrow L_s^{r,q_1}(\rn)$ yields that 
\begin{equation*}
\big\Vert \sigma(2^j\cdot)\widehat{\Psi}\big\Vert_{L_s^{r,q_1}(\rn)}\lesssim \big\Vert \sigma(2^j\cdot)\widehat{\Psi}\big\Vert_{L_{s}^{r,q_2}(\rn)} \quad \text{~uniformly in }~ j. 
\end{equation*}
Consequently, we may replace $L_s^{\tau^{(s,p)},\min{(1,p)}}(\rn)$ in Theorem \ref{mainpositive} by $L_s^{r,q}(\rn)$ for $r> \tau^{(s,p)}$  and $0<q\le \infty$, or by $L_s^{\tau^{(s,p)},q}(\rn)$ for $0<q< \min{(1,p)}$.

The second main result of this paper is the sharpness of the parameters $ \tau^{(s,p)}$ and 
$ \min(1,p)$. That is,  Theorem \ref{mainpositive} is sharp in the sense that $\tau^{(s,p)}$ cannot be replaced by any smaller number $r$, and if $r=\tau^{(s,p)}$, then $\min{(1,p)}$ cannot be replaced 
by any larger number $q$.
\begin{theorem}\label{mainnegative}
Let $0<p<\infty$ and $|n/p-n/2|<s<n/\min{(1,p)}$.
\begin{enumerate}
\item For any $0<r<\tau^{(s,p)}$ and $0<q\le \infty$, there exists a function $\sigma$ that satisfies
\begin{equation*}
\sup_{j\in\zz}{\big\Vert  \sigma(2^j\cdot)\widehat{\Psi} \big\Vert_{L_{s}^{r,q}(\rn)}}<\infty 
\end{equation*}
such that $T_{\sigma}$ is unbounded on $H^p(\rn)$.
\item For any $q>\min{(1,p)}$, there exists a function $\sigma$ that satisfies
\begin{equation*}
\sup_{j\in\zz}{\big\Vert  \sigma(2^j\cdot)\widehat{\Psi} \big\Vert_{L_{s}^{\tau^{(s,p)},q}(\rn)}}<\infty
\end{equation*}
such that $T_{\sigma}$ is unbounded on $H^p(\rn)$.

\end{enumerate}

\end{theorem}
We remark that the first assertion follows immediately from the second one which is the end-point case $r=\tau^{(s,p)}$, thanks to (\ref{lorentzembedding}). Therefore, only the second statement will be considered in the proof of Theorem \ref{mainnegative}.

The paper is organized as follows. Section \ref{preliminary} is dedicated to preliminaries, mostly extensions of inequalities in Lebesgue spaces to Lorentz spaces thanks to a real interpolation technique. We address the case $0<p\le 1$ of Theorem \ref{mainpositive} in Section \ref{proofpositive} and the proof of Theorem \ref{mainnegative} is given in Section \ref{proofnegative}. In the appendix, a complex interpolation method is   discussed whose purpose is to establish the $L^p$-boundedness for $1<p<2$.

\section{Preliminaries}\label{preliminary}

The Lorentz spaces are generalization of Lebesgue spaces and occur as intermediate spaces for the real interpolation, the so called $K$-method.
For $0<p,p_0,p_1<\infty$, $0<r\le \infty$, and $0<\theta<1$ satisfying $p_0\not= p_1$ and $1/p=(1-\theta)/p_0+\theta/ p_1$, we have 
\begin{equation}\label{lebesgueinter}
(L^{p_0}(\rn),L^{p_1}(\rn))_{\theta,r}=L^{p,r}(\rn).
\end{equation}
This remains valid for vector-valued spaces.
For $0<p,p_0,p_1<\infty$, $0<q,r\le \infty$, and $0<\theta<1$ satisfying $p_0\not= p_1$ and $1/p=(1-\theta)/p_0+\theta/ p_1$,
\begin{equation}\label{vectorinter}
\big(L^{p_0}(\ell^q),L^{p_1}(\ell^q)\big)_{\theta,r}=L^{p,r}(\ell^q), \quad \big(\ell^q(L^{p_0}),\ell^q(L^{p_1})\big)_{\theta,r}=\ell^q(L^{p,r}).
\end{equation}
We remark that $\big((L^{p_0}(\ell^{q_0}),L^{p_1}(\ell^{q_1})\big)_{\theta,r}\not=L^{p,r}(\ell^q)$, $\big(\ell^{q_0}(L^{p_0}),\ell^{q_1}(L^{p_1})\big)_{\theta,r}\not=\ell^q(L^{p,r}) $ for $q_0\not= q_1$ with $1/q=(1-\theta)/q_0+\theta/q_1$.
See \cite{Be_Sh, Be_Lo, Cw, Fe_Ri_Sa} for more details.

Then many inequalities in Lebesgue spaces can be extended to Lorentz spaces from the following real interpolation method; see \cite{Be_Sh, Be_Lo, Fe_Ri_Sa, Hol}.
\begin{customproposition}{B}\label{interpolation}
Let $\mathcal{A}$ and $\mathcal{B}$ be two topological vector spaces.
Suppose $(A_0, A_1)$ and $(B_0,B_1)$ be couples of quasi-normed spaces continuously embedded into $\mathcal{A}$ and $\mathcal{B}$, respectively.
Let $0<\theta<1$ and $0<r\le \infty$.
If $T$ is a linear operator such that 
\begin{equation*}
T: A_0 \to B_0, \qquad T: A_1\to B_1,
\end{equation*} with the quasi-norms $M_0$ and $M_1$, respectively, then
\begin{equation*}
T:(A_0,A_1)_{\theta,r}\to (B_0,B_1)_{\theta,r}
\end{equation*} is also continuous, and for its quasi-norm we have 
\begin{equation*}
\Vert T\Vert_{(A_0,A_1)_{\theta,r}\to (B_0,B_1)_{\theta,r}}\le M_0^{1-\theta}M_1^{\theta}.
\end{equation*}
\end{customproposition}

We apply   Proposition \ref{interpolation}   to extend Young's inequality, the Hausdorff-Young inequality, Minkowski's  inequality,  and the Kato-Ponce   inequality to Lorentz spaces.

\begin{lemma}\label{young}
Let $1<p\le r<\infty$, $1\le q<r$, and $0<t\le \infty$ satisfy $1/r+1=1/p+1/q$.
Then
\begin{equation*}
\Vert f\ast g\Vert_{L^{r,t}(\rn)}\le \Vert f\Vert_{L^{p,t}(\rn)} \Vert g\Vert_{L^q(\rn)}
\end{equation*}
for all $f,g\in \mathscr S(\rn)$.
\end{lemma}
\begin{proof}
For a fixed $g\in \mathscr S(\rn)$, we define the linear operator $T_g$  by 
\begin{equation*}
T_gf:=f\ast g.
\end{equation*}
Choose $r_1$, $\theta$, and $p_1$ such that $r<r_1<\infty$, $0<\theta<1$, $p<p_1<\infty$, $1/r=(1-\theta)/q+\theta/r_1$, and $1/r_1+1=1/p_1+1/q$. Then note that $1/p=1-\theta+\theta/p_1$.
By using Young inequality, we obtain that
\begin{equation*}
\Vert T_gf\Vert_{L^q(\rn)}\le \Vert g\Vert_{L^q}\Vert f\Vert_{L^1(\rn)}
\end{equation*}
and
\begin{equation*}
\Vert T_gf\Vert_{L^{r_1}(\rn)}\le \Vert g\Vert_{L^q}\Vert f\Vert_{L^{p_1}(\rn)}.
\end{equation*}
Then Proposition \ref{interpolation} with (\ref{lebesgueinter}) completes the proof.
\end{proof}

\begin{lemma}\label{hausyoung}
Let $2<p<\infty$ and $0<r\le \infty$.
Then 
\begin{equation*}
\Vert \widehat{f}\Vert_{L^{p,r}(\rn)}\le \Vert f\Vert_{L^{p',r}(\rn)}
\end{equation*}
where $1/p+1/p'=1$.

\end{lemma}
\begin{proof}
It follows immediately from Hausdorff-Young inequality and Proposition \ref{interpolation} with (\ref{lebesgueinter}).
\end{proof}

\begin{lemma}\label{katoponce}
Let $1<p<\infty$, $0<r\le \infty$, and $s>0$. For any $\vartheta\in \mathscr S(\rn)$, we have
\begin{equation}\label{katoponcestate}
\Vert  \vartheta\cdot f \Vert_{L^{p,r}_{s}(\rn)}\lesssim_{n,s,p,r,\vartheta} \Vert f\Vert_{L_s^{p,r}(\rn)}.
\end{equation}

\end{lemma}
\begin{proof}
Pick $p_0$, $p_1$ satisfying $1<p_0<p<p_1<\infty$ and let $T$ be the linear operator defined by
\begin{equation*}
Tf:=(I-\Delta)^{s/2}\big(\vartheta \cdot (I-\Delta)^{-s/2}f \big).
\end{equation*}
Then we apply the Kato-Ponce inequality \cite{Ka_Po} to obtain
\begin{equation*}
\Vert Tf\Vert_{L^{p_j}}\lesssim \Vert f\Vert_{L^{p_j}} \quad \text{ for }~ j=0,1.
\end{equation*}
Then (\ref{katoponcestate}) follows from Proposition \ref{interpolation} and (\ref{lebesgueinter}).
\end{proof}

\begin{lemma}\label{minkowski}
Let $1\le q<p<\infty$ and $0<r\le \infty$.
Then
\begin{equation*}
\Big\Vert \Big(\sum_{k\in\zz}{|f_k|^q} \Big)^{1/q}\Big\Vert_{L^{p,r}(\rn)}\lesssim \Big( \sum_{k\in\zz}{\Vert f_k\Vert_{L^{p,r}(\rn)}^q}\Big)^{1/q}
\end{equation*}

\end{lemma}
\begin{proof}
We select $p_1>0$ and $0<\theta<1$ so that $p<p_1<\infty$ and $1/p=(1-\theta)/p_1+\theta/q$.
Using Minkowski inequality, 
$\big\Vert \big\{f_k\big\}_{k\in\zz}\big\Vert_{L^{p_1}(\ell^q)}\lesssim \big\Vert \big\{f_k\big\}_{k\in\zz}\big\Vert_{\ell^q(L^{p_1})}$ and we interpolate this with $\big\Vert \big\{f_k\big\}_{k\in\zz}\big\Vert_{L^{q}(\ell^q)}=\big\Vert \big\{f_k\big\}_{k\in\zz}\big\Vert_{\ell^{q}(L^q)}$ to obtain
\begin{equation*}
\big\Vert \big\{f_k\big\}_{k\in\zz}\big\Vert_{(L^{p_1}(\ell^q),L^q(\ell^q))_{\theta,r}}\lesssim \big\Vert \big\{f_k\big\}_{k\in\zz}\big\Vert_{(\ell^q(L^{p_1},\ell^q(L^q)))_{\theta,r}}.
\end{equation*}
Then the proof is completed in view of (\ref{vectorinter}).
\end{proof}

The next ingredient we need is H\"older's inequality in Lorentz spaces, which is an  immediate  consequence of the Hardy-Littlewood inequality
\begin{equation*}
\int_{\rn}{|f(x)g(x)|}dx\le \int_0^{\infty}{f^*(t)g^*(t)}dt
\end{equation*} and H\"older's inequality for Lebesgue spaces.
\begin{lemma}\label{holder}
Let $1<p<\infty$ and $1\le q\le \infty$.
Then
\begin{equation*}
 \int_{\rn}{\big|f(x)g(x)\big|}dx\le \Vert f\Vert_{L^{p,q}(\rn)}\Vert g\Vert_{L^{p',q'}(\rn)}
\end{equation*}
where $1/p+1/p'=1/q+1/q'=1$.
\end{lemma}

The following Lorentz space variant of the Sobolev embedding theorem can be easily obtained from the
classical  Sobolev embedding theorem  combined with the Marcinkiewicz interpolation theorem; the proof is omitted. 

\begin{lemma}\label{embedding}
Let $s_0,s_1\in\rr$, $1<p_0,p_1<\infty$, and $0<r_0,r_1\le \infty$.
Then the embedding \begin{equation*}
L^{p_0,r_0}_{s_0}(\rn)\hookrightarrow L^{p_1,r_1}_{s_1}(\rn)
\end{equation*}
holds if $p_0=p_1$, $s_0\ge s_1$, $r_0\le r_1$, or 
 if $s_0-s_1=n/p_0-n/p_1>0$.
\end{lemma}

We remark that a generalization of the preceding lemma can be found in the recent work of Seeger and Trebels 
\cite{Se_Tr}.

 Finally, we   describe the behavior of decreasing rearrangement of radial functions.
\begin{lemma}\label{radialrearrangement}
Suppose $f$ is a radial function with $f(x)=g(|x|)$ for $x\in \rn$.
Then
\begin{equation*}
f^*(t)=g^*\big(( {t}/{\Omega_n})^{1/n}\big)
\end{equation*}
where $\Omega_n$ stands for the volume of the unit ball in $\rn$.
\end{lemma}
\begin{proof}
We observe that
\begin{align*}
d_f(s)=\big|\big\{x\in \rn:|f(x)|>s \big\} \big|&=\big|\big\{r\theta \in \rn:|g(r)|>s, \theta\in \mathbb S^{n-1} \big\} \big|\\
&=\Omega_n\big|\big\{r>0: |g(r)|>s \big\} \big|^n\\
&=\Omega_n \big( d_g(s)\big)^n
\end{align*}
and this proves that
\begin{align*}
f^*(t)=\inf\big\{s>0:d_f(s)\le t\big\}&=\inf\big\{s>0:\Omega_n\big( d_g(s)\big)^n\le t\big\}\\
&=\inf\big\{s>0:d_g(s)\le ( {t}/{\Omega_n})^{1/n}\big\}\\
&=g^*\big((t/\Omega_n)^{1/n}\big).
\end{align*}
\end{proof}

\section{Proof of Theorem \ref{mainpositive}}\label{proofpositive}

The set of Schwartz functions whose Fourier transform is compactly supported away from the origin is dense in $H^p(\rn)$; 
this is a consequence  of Littlewood-Paley theory for $H^p$ as one can approximate $f\in  H^p$ by  
\begin{equation*}
f^{(N)}:=\sum_{k=-N}^{N}{2^{kn}\Psi(2^k\cdot)\ast f} \to f  \quad \text{ in }~~ H^p(\rn) \qquad \text{ as }~~ N\to \infty.
\end{equation*}
See \cite{Tr1} for more details. Thus we may work with such Schwartz functions. 
Let $f$ be a Schwartz function with compact support away from the origin in frequency space and suppose $\sigma\in L^{\infty}(\rn)$ satisfies 
\begin{equation*}
\sup_{j\in \zz}{\big\Vert \sigma(2^j\cdot)\widehat{\Psi}\big\Vert_{L^{\tau^{(s,p)},p}_s(\rn)}}<\infty.
\end{equation*}
Let $\Lambda\in \mathscr S(\rn)$ have the properties that $\textup{Supp}(\Lambda)\subset \{\xi\in \rn: |\xi|\le 1\}$ and $\int_{\rn}{\Lambda(\xi)}d\xi=1$.
For $0<\epsilon<1/100$, we introduce 
\begin{equation*}
\sigma^{\epsilon}(\xi):=\sum_{j\in \zz}{\big( \sigma\widehat{\Psi}(\cdot/2^j)\big)\ast \Lambda^{j,\epsilon} (\xi)}
\end{equation*} where $\Lambda^{j,\epsilon}:={(2^j\epsilon)^{-n}}\Lambda(\cdot/2^j\epsilon)$.
Then since $\widehat{f}$ has compact support away from the origin,
\begin{equation*}
T_{\sigma^{\epsilon}}f=\sum_{j\in\zz}{{\Big( \big[\big( \sigma\widehat{\Psi}(\cdot/2^j)\big)\ast \Lambda^{j,\epsilon} \big] \widehat{f} \,\Big)^{\vee}}}
\end{equation*} is a finite sum and thus, using the argument of approximation of identity, for each $k\in \zz$
\begin{equation*}
\lim_{\epsilon\to 0} \Phi_k\ast \big(T_{\sigma^{\epsilon}}f\big)(x)=\Phi_k\ast \big(T_{\sigma}f\big)(x).
\end{equation*}
This proves that
\begin{align*}
\big\Vert T_{\sigma}f\big\Vert_{H^p(\rn)}\le \big\Vert \liminf_{\epsilon\to 0}{\sup_{k\in\zz}{\big| \Phi_k\ast (T_{\sigma^{\epsilon}}f)\big|}}\big\Vert_{L^p(\rn)}\le \liminf_{\epsilon\to 0}\big\Vert  T_{\sigma^{\epsilon}}f\big\Vert_{H^p(\rn)}
\end{align*}
where we applied Fatou's lemma in the last inequality.
Therefore, it suffices to show that
\begin{equation}\label{reduce1}
\Vert T_{\sigma^{\epsilon}}f\Vert_{H^p(\rn)}\lesssim \sup_{j\in \zz}{\big\Vert \sigma(2^j\cdot)\widehat{\Psi}\big\Vert_{L^{\tau^{(s,p)},p}_s(\rn)}}\Vert f\Vert_{H^p(\rn)}, \qquad \text{ uniformly in }~ \epsilon.
\end{equation}

Now there exist a sequence of $L^{\infty}$-atoms $\{a_l\}_{l=1}^{\infty}$ for $H^p(\rn)$, and a sequence of scalars $\{\lambda_l\}_{l=1}^{\infty}$ so that
\begin{equation*}
f=\sum_{l=1}^{\infty}{\lambda_l a_l} \qquad \textup{in $\,\,\mathscr S'$}
\end{equation*} and
\begin{equation*}
\Big( \sum_{l=1}^{\infty}{|\lambda_l|^{p}}\Big)^{1/p}\approx \Vert f\Vert_{H^p(\rn)},
\end{equation*}
where $L^{\infty}$-atom $a_l$ for $H^p(\rn)$ means that there exists a cube $Q_l$ such that $a_l$ is supported in $Q_l$, $| a_l | \le |Q_l|^{-1/p}$, and $\int_{\rn}{x^{\gamma} a_l(x)}dx=0$ for all multi-indices $\gamma$ with $|\gamma|\le [n/p-n]$.

We note that $T_{\sigma^{\epsilon}}$ maps $\mathscr S(\rn)$ to itself, which implies that $T_{\sigma^{\epsilon}}$ is well-defined on $\mathscr S'(\rn)$ using duality argument, and actually,
$T_{\sigma^{\epsilon}}:\mathscr S'(\rn)\to \mathscr  S'(\rn)$.
This yields that
\begin{equation*}
T_{\sigma^{\epsilon}}f=\sum_{l=1}^{\infty}{\lambda_l (T_{\sigma^{\epsilon}}a_l)} \qquad \text{ in the sense of tempered distribution}.
\end{equation*} 
Hence we have 
\begin{equation*}
\Vert T_{\sigma^{\epsilon}}f\Vert_{H^p(\rn)}\le \Big( \sum_{l=1}^{\infty}{|\lambda_l|^p\big\Vert T_{\sigma^{\epsilon}}a_l\big\Vert_{H^p(\rn)}^p}\Big)^{1/p},
\end{equation*} using the subadditive property of $\Vert \cdot \Vert_{H^p(\rn)}^p$.

Moreover, due to support assumptions and  dilations, for each $j\in \zz$, we have 
\begin{equation*}
\sigma^{\epsilon}(2^j\xi)\widehat{\Psi}(\xi)=\sum_{l=j-2}^{j+2}{\big( \sigma\widehat{\Psi}(\cdot/2^l)\big)\ast \Lambda^{l,\epsilon} (2^j\xi)}\widehat{\Psi}(\xi)=\sum_{l=-2}^{2}{\big( \sigma(2^j\cdot)\widehat{\Psi}(\cdot/2^l)\big)\ast \Lambda^{l,\epsilon} (\xi)}\widehat{\Psi}(\xi),
\end{equation*}
from which it follows 
\begin{align*}
&\sup_{j\in\zz}{\big\Vert \big(\sigma^{\epsilon}(2^j\cdot)\widehat{\Psi} \big)\big\Vert_{L_s^{\tau^{(s,p)},p}(\rn)}}\lesssim {\sum_{l=-2}^{2}{\sup_{j\in\zz}\Big\Vert (I-\Delta)^{s/2}\Big(\big( \sigma(2^j\cdot)\widehat{\Psi}(\cdot/2^{l})\big)\ast \Lambda^{l,\epsilon} \Big)\Big\Vert_{L^{\tau^{(s,p)},p}(\rn)}}}\\
&\lesssim  \sum_{l=-2}^{2}\sup_{j\in\zz}{{\big\Vert   \sigma(2^j\cdot)\widehat{\Psi}(\cdot/2^{l}) \big\Vert_{L^{\tau^{(s,p)},p}_s(\rn)}}}\le \sum_{l=-2}^{2}C_l\sup_{j\in\zz}{{\big\Vert   \sigma(2^{j+l}\cdot)\widehat{\Psi} \big\Vert_{L^{\tau^{(s,p)},p}_s(\rn)}}}\\
& \lesssim \sup_{j\in\zz}{\big\Vert \sigma(2^j\cdot)\widehat{\Psi}\big\Vert_{L_s^{\tau^{(s,p)},p}(\rn)}}
\end{align*}
uniformly in $\epsilon$; here we applied Lemmas \ref{katoponce} and \ref{young} combined with
 the fact that $\Vert \Lambda^{l,\epsilon}\Vert_{L^1(\rn)}=\Vert \Lambda \Vert_{L^1(\rn)}$.

Therefore, the proof of (\ref{reduce1}) is reduced to the following proposition.

\begin{proposition}
Let $0<p\le 1$ and $a$ be a $H^p$-atom, associated with a cube $Q$ in $\rn$.
Then we have \begin{equation*}
\Vert T_{\sigma}a\Vert_{H^p(\rn)}\lesssim \sup_{j\in\zz}{\big\Vert  \sigma(2^j\cdot)\widehat{\Psi} \big\Vert_{L_{s}^{\tau^{(s,p)},p}(\rn)}}
\end{equation*} where the constant in the inequality is independent of $\sigma$ and $a$.

\end{proposition}

\begin{proof}

Introducing the function $\Theta$ satisfying $\widehat{\Theta}(\xi):=\widehat{\Psi}(\xi/2)+\widehat{\Psi}(\xi)+\widehat{\Psi}(2\xi)$ so that $\widehat{\Theta}=1$ on the support of $\widehat{\Psi}$,
let  $\mathcal{L}_j$ and $\mathcal{L}_j^{\Theta}$ be the Littlewood-Paley operators associated with $\Psi$ and $\Theta$, respectively.
Let $Q^*$ and $Q^{**}$ denote the concentric dilates of $Q$ with side length $10l(Q)$ and $100l(Q)$, respectively.
Then we write
\begin{align*}
\Vert T_{\sigma}a\Vert_{H^p(\rn)}&\approx \Big\Vert \Big( \sum_{j\in\zz}{| \mathcal{L}_jT_{\sigma}a |^2}\Big)^{1/2}\Big\Vert_{L^p(\rn)}\\
 &\lesssim_p \Big\Vert \Big( \sum_{j\in\zz}{| \mathcal{L}_jT_{\sigma}a |^2}\Big)^{1/2}\Big\Vert_{L^p(Q^{**})}+\Big\Vert \Big( \sum_{j\in\zz}{| \mathcal{L}_jT_{\sigma}a |^2}\Big)^{1/2}\Big\Vert_{L^p((Q^{**})^c)}.
\end{align*}
In view of H\"older's inequality, the first part is controlled by
\begin{equation*}
|Q^{**}|^{1/p-1/2} \Big\Vert \Big( \sum_{j\in\zz}{| \mathcal{L}_jT_{\sigma}a |^2}\Big)^{1/2}\Big\Vert_{L^2(\rn)}\lesssim_n |Q|^{1/p-1/2}\Vert T_{\sigma}a\Vert_{L^2(\rn)}
\end{equation*}
and we see that
\begin{equation*}
\Vert T_{\sigma}a\Vert_{L^2(\rn)}\le \Vert \sigma\Vert_{L^{\infty}(\rn)}\Vert a\Vert_{L^2(\rn)}\le  \sup_{j\in\zz}{\big\Vert \sigma(2^j\cdot)\widehat{\Psi}\big\Vert_{L^{\infty}(\rn)}} |Q|^{-(1/p-1/2)}.
\end{equation*}
Now using Lemma \ref{holder}, \ref{hausyoung}, and \ref{embedding} with $1<\tau^{(s,p)}<2$, we obtain
\begin{align*}
{\big\Vert \sigma(2^j\cdot)\widehat{\Psi}\big\Vert_{L^{\infty}(\rn)}}&\le {\big\Vert \big(\sigma(2^j\cdot)\widehat{\Psi} \big)^{\vee}\big\Vert_{L^1(\rn)}}\\
&\lesssim {\big\Vert \big( 1+4\pi^2|\cdot|^2\big)^{(s-(n/p-n))/2}\big(\sigma(2^j\cdot)\widehat{\Psi} \big)^{\vee} \big\Vert_{L^{(\tau^{(s,p)})',1}(\rn)}}\\
 &\le \big\Vert \sigma(2^j\cdot)\widehat{\Psi}\big\Vert_{L^{\tau^{(s,p)},1}_{s-(n/p-n)}(\rn)} \lesssim \big\Vert \sigma(2^j\cdot)\widehat{\Psi}\big\Vert_{L^{\tau^{(s,p)},p}_{s}(\rn)},
\end{align*} which finishes the proof of 
\begin{equation*}
\Big\Vert \Big( \sum_{j\in\zz}{| \mathcal{L}_jT_{\sigma}a |^2}\Big)^{1/2}\Big\Vert_{L^p(Q^{**})} \lesssim \sup_{j\in \zz}{\Vert \sigma(2^j\cdot)\widehat{\Psi}\Vert_{L^{\tau^{(s,p)},p}_{s}(\rn)}}.
\end{equation*}

To verify
\begin{equation}\label{outest}
\Big\Vert \Big( \sum_{j\in\zz}{| \mathcal{L}_jT_{\sigma}a |^2}\Big)^{1/2}\Big\Vert_{L^p((Q^{**})^c)}\lesssim \sup_{j\in\zz}{\big\Vert \sigma(2^j\cdot)\widehat{\Psi}\big\Vert_{L_{s}^{\tau^{(s,p)},p}(\rn)}},
\end{equation}
we notice that $\mathcal{L}_jT_{\sigma}a(x)$ can be written as $\big(\sigma \widehat{\Psi}(\cdot/2^j)\big)^{\vee}\ast (\mathcal{L}_j^{\Theta}a)(x)$. 
We decompose the left-hand side of (\ref{outest}) to
\begin{equation*}
\mathcal{I}:=\Big\Vert \Big( \sum_{j: 2^jl(Q)<1}{\big| (\sigma \widehat{\Psi}(\cdot/2^j))^{\vee}\ast (\mathcal{L}_j^{\Theta}a) \big|^2}\Big)^{1/2}\Big\Vert_{L^p((Q^{**})^c)}
\end{equation*} and
\begin{equation*}
\mathcal{J}:=\Big\Vert \Big( \sum_{j: 2^jl(Q)\ge 1}{\big| (\sigma \widehat{\Psi}(\cdot/2^j))^{\vee}\ast (\mathcal{L}_j^{\Theta}a) \big|^2}\Big)^{1/2}\Big\Vert_{L^p((Q^{**})^c)}.
\end{equation*}

In view of the embedding $\ell^p\hookrightarrow \ell^2$
\begin{equation*}
\mathcal{I}\le \Big(\sum_{j:2^jl(Q)<1}{\big\Vert (\sigma \widehat{\Psi}(\cdot/2^j))^{\vee}\ast (\mathcal{L}_j^{\Theta}a) \big\Vert_{L^p(\rn)}^p} \Big)^{1/p}
\end{equation*}
and Bernstein's inequality, we obtain  
\begin{equation*}
\big\Vert (\sigma \widehat{\Psi}(\cdot/2^j))^{\vee}\ast (\mathcal{L}_j^{\Theta}a) \big\Vert_{L^p(\rn)}\lesssim 2^{jn(1/p-1)}\big\Vert (\sigma \widehat{\Psi}(\cdot/2^j))^{\vee}\big\Vert_{L^p(\rn)} \Vert \mathcal{L}_j^{\Theta}a\Vert_{L^p(\rn)}.
\end{equation*}
Using dilation, Lemma \ref{holder} and \ref{hausyoung}, we have 
\begin{align}
2^{jn(1/p-1)}\Vert (\sigma \widehat{\Psi}(\cdot/2^j))^{\vee}\Vert_{L^p(\rn)}&=\Big(\int_{\rn}{ \big|\big( \sigma(2^j\cdot ) \widehat{\Psi}\big)^{\vee}(x) \big|^p}dx\Big)^{1/p} \nonumber \\
&\lesssim \Big\Vert \big| \big(1+4\pi^2|\cdot|^2\big)^{s/2}   \big(\sigma(2^j\cdot) \widehat{\Psi}\big)^{\vee}\big|^p  \Big\Vert_{L^{(n/(sp))',1}(\rn)}^{1/p} \nonumber \\
&=\Big\Vert  \big(1+4\pi^2|\cdot|^2\big)^{s/2}   \big(\sigma(2^j\cdot) \widehat{\Psi}\big)^{\vee}  \Big\Vert_{L^{p(n/(sp))',p}(\rn)} \nonumber \\
&\le \big\Vert \sigma(2^j\cdot)\widehat{\Psi}\big\Vert_{L^{\tau^{(s,p)},p}_s(\rn)} \label{symbolest}
\end{align}
since $2<p(n/(sp))'<\infty$ and $\tau^{(s,p)}=\big(p(n/(sp))'\big)'$.
 Moreover,  for any $M>0$
 \begin{equation*}
 |\mathcal{L}_j^{\Theta}a(x)|\lesssim_M |Q|^{1-1/p}\big( 2^jl(Q)\big)^{[n/p-n]+1}\dfrac{2^{jn}}{(1+2^j|x-c_Q|)^M},
 \end{equation*}
 using standard arguments in \cite[Appendix B]{Gr2} with $2^jl(Q)<1$ and the fact that 
 \begin{equation*}
 |a(x)|\lesssim_{n,M} |Q|^{-1/p}\dfrac{1}{\big(1+|x-c_Q|/l(Q)\big)^M}, \qquad \int_{\rn}{x^{\alpha}a(x)}dx=0 ~\text{ for }~ |\alpha|\le [n/p-n],
 \end{equation*}
\begin{equation*}
\big| \partial^{\alpha}\big(2^{jn}\Psi(2^j\cdot) \big)(x)\big|\lesssim 2^{j|\alpha|}2^{jn}\dfrac{1}{(1+2^j|x|)^M} ~\text{ for }~ \alpha\in\zn
\end{equation*}
where $c_Q$ denotes the center of $Q$. Selecting $M>n/p$, we have
\begin{equation*}
\Vert \mathcal{L}_ja\Vert_{L^p}\lesssim \big( 2^jl(Q)\big)^{[n/p]+1-n/p}
\end{equation*}
and thus
\begin{align*}
\mathcal{I}&\lesssim \sup_{j\in\zz}{\big\Vert \sigma(2^j\cdot)\widehat{\Psi}\big\Vert_{L^{\tau^{(s,p)},p}_s(\rn)}}\Big(\sum_{j:2^jl(Q)<1}{\big(2^jl(Q) \big)^{p([n/p]+1-n/p)}} \Big)^{1/p}\\
&\lesssim \sup_{j\in\zz}{\big\Vert \sigma(2^j\cdot)\widehat{\Psi}\big\Vert_{L^{\tau^{(s,p)},p}_s(\rn)}},
\end{align*} 
since $[n/p]+1-n/p>0$.

To estimate $\mathcal{J}$ we further separate into two terms
\begin{equation*}
\mathcal{J}_1:=\Big\Vert \Big( \sum_{j: 2^jl(Q)\ge 1}{\big| \big(\sigma \widehat{\Psi}(\cdot/2^j)\big)^{\vee}\ast \big(\chi_{(Q^*)^c}\mathcal{L}_j^{\Theta}a\big) \big|^2}\Big)^{1/2}\Big\Vert_{L^p((Q^{**})^c)}
\end{equation*} and
\begin{equation*}
\mathcal{J}_2:=\Big\Vert \Big( \sum_{j: 2^jl(Q)\ge 1}{\big| \big(\sigma \widehat{\Psi}(\cdot/2^j)\big)^{\vee}\ast \big(\chi_{Q^*}\mathcal{L}_j^{\Theta}a\big) \big|^2}\Big)^{1/2}\Big\Vert_{L^p((Q^{**})^c)}.
\end{equation*}
Using the embedding $\ell^p\hookrightarrow \ell^2$, Bernstein inequality with 
\begin{equation*}
 \big(\sigma \widehat{\Psi}(\cdot/2^j)\big)^{\vee}\ast \big(\chi_{(Q^*)^c}\mathcal{L}_j^{\Theta}a\big)(x)= \big(\sigma \widehat{\Psi}(\cdot/2^j)\big)^{\vee}\ast \big[\mathcal{L}_j^{\Theta}\big(\chi_{(Q^*)^c}\mathcal{L}_j^{\Theta}a\big)\big](x),
 \end{equation*}
  and the inequality (\ref{symbolest}), we have
\begin{equation*}
\mathcal{J}_1\lesssim \sup_{j\in \zz}{     \big\Vert \sigma(2^j\cdot)\widehat{\Psi}\big\Vert_{L^{\tau^{(s,p)},p}_s(\rn)}     } \Big( \sum_{j: 2^jl(Q)\ge 1}{\big\Vert \mathcal{L}_j^{\Theta}\big(\chi_{(Q^*)^c}\mathcal{L}_j^{\Theta}a\big)\big\Vert_{L^p(\rn)}^p}\Big)^{1/p}.
\end{equation*}
We see that for $x\in (Q^*)^c$ and $M>n/p$
\begin{align*}
|\mathcal{L}_j^{\Theta}a(x)|&\lesssim_M |Q|^{-1/p}\int_{y\in Q}{\dfrac{2^{jn}}{(1+2^j|x-y|)^{2M}}}dy\lesssim_M |Q|^{-1/p}\dfrac{1}{(2^j|x-c_Q|)^M}\\
 &\lesssim_M |Q|^{-1/p}(2^jl(Q))^{-M}\dfrac{1}{(1+|x-c_Q|/l(Q))^M}
\end{align*} since $|x-y|\ge \frac{9}{10}|x-c_Q|$.
Then
\begin{align*}
&\big\Vert \mathcal{L}_j^{\Theta}\big(\chi_{(Q^*)^c}\mathcal{L}_j^{\Theta}a\big)\big\Vert_{L^p(\rn)}\\
&\lesssim |Q|^{-1/p}(2^jl(Q))^{-M}\bigg[ \int_{\rn}{\Big(\int_{\rn}{|2^{jn}\Theta(2^j(x-y))|\dfrac{1}{(1+|x-c_Q|/l(Q))^M}}dy \Big)^p}dx\bigg]^{1/p}.
\end{align*}
 Standard manipulations with $2^jl(Q)\ge 1$ in \cite[Appendix B]{Gr2} yield that the last expression is less than a constant times
\begin{align*}
|Q|^{-1/p}(2^jl(Q))^{-M}\Big(\int_{\rn}{\dfrac{1}{(1+|x-c_Q|/l(Q))^{Mp}}}dx \Big)^{1/p}\lesssim (2^jl(Q))^{-M}.
\end{align*}
Accordingly,
\begin{equation*}
\mathcal{J}_1\lesssim \sup_{j\in \zz}{     \big\Vert \sigma(2^j\cdot)\widehat{\Psi}\big\Vert_{L^{\tau^{(s,p)},p}_s(\rn)}     } \Big(\sum_{k:2^kl(Q)\ge 1 }{ (2^kl(Q))^{-Mp}} \Big)^{1/p}\lesssim \sup_{j\in \zz}{     \big\Vert \sigma(2^j\cdot)\widehat{\Psi}\big\Vert_{L^{\tau^{(s,p)},p}_s(\rn)}     }. 
\end{equation*}

We now consider $\mathcal{J}_2$. Choose $n/p-n/2<N<s$ so that $n/2<Np<sp<n$ and $2<p\big(n/(Np)\big)'<\infty$.
For notational convenience we write
\begin{equation*}
\mathcal{E}_j^N\sigma(x):=\big( 1+4\pi^2(2^j|x|)^2\big)^{N/2}\big( \sigma \widehat{\Psi}(\cdot/2^j)\big)^{\vee}(x).
\end{equation*}
Observe that if $x\in (Q^{**})^c$ and $y\in Q^*$, then $|x-c_Q|\lesssim |x-y|$ and thus
\begin{equation*}
|x-c_Q|^N\big|\big(\sigma \widehat{\Psi}(\cdot/2^j) \big)^{\vee}\ast\big(\chi_{Q^*}\mathcal{L}_j^{\Theta}a\big)(x) \big|\lesssim 2^{-jN}\big|  \mathcal{E}_j^N\sigma     \big|\ast \big| \chi_{Q^*}\mathcal{L}_j^{\Theta}a\big|(x).
\end{equation*}
This proves that $\mathcal{J}_2$ is less than a constant times
\begin{align*}
 &\Big\Vert \dfrac{1}{|x-c_Q|^N}\Big( \sum_{j:2^jl(Q)\ge 1}{2^{-2jN}\Big( \big|\mathcal{E}_j^N\sigma \big|\ast \big| \chi_{Q^*}\mathcal{L}_j^{\Theta}a\big| \Big)^2}\Big)^{1/2}\Big\Vert_{L^p((Q^{**})^c)}\\
 &\lesssim\Big\Vert \Big( \sum_{j:2^jl(Q)\ge 1}{2^{-2jN}\Big( \big|\mathcal{E}_j^N\sigma \big|\ast \big| \chi_{Q^*}\mathcal{L}_j^{\Theta}a\big| \Big)^2} \Big)^{p/2}\Big\Vert_{L^{(n/(Np))',1}(\rn)}^{1/p}\\
 &=\Big\Vert \Big( \sum_{j:2^jl(Q)\ge 1}{2^{-2jN}\Big( \big|\mathcal{E}_j^N\sigma\big|\ast \big| \chi_{Q^*}\mathcal{L}_j^{\Theta}a\big| \Big)^2} \Big)^{1/2}\Big\Vert_{L^{p(n/(Np))',p}(\rn)}, 
\end{align*}
where  we made use of  Lemma \ref{holder} with $n/(Np)>1$.
Now using Lemma \ref{minkowski} with $p(n/(Np))'>2$,
the preceding expression is dominated by a constant multiple of 
\begin{equation*}
\Big(\sum_{j:2^jl(Q)\ge 1}{2^{-2jN}\Big\Vert \big|\mathcal{E}_j^N\sigma \big|\ast \big| \chi_{Q^*}\mathcal{L}_j^{\Theta}a\big| \Big\Vert_{L^{p(n/(Np))',p}(\rn)}^2} \Big)^{1/2}
\end{equation*}
and Lemma \ref{young} yields that
\begin{equation*}
\Big\Vert \big|\mathcal{E}_j^N\sigma \big|\ast \big| \chi_{Q^*}\mathcal{L}_j^{\Theta}a\big| \Big\Vert_{L^{p(n/(Np))',p}(\rn)}\lesssim \big\Vert \mathcal{E}_j^N\sigma  \big\Vert_{L^{p(n/(Np))',p}(\rn)}\Vert \mathcal{L}_j^{\Theta}a\Vert_{L^1(Q^*)}.
\end{equation*}
We see that
\begin{align*}
\big\Vert \mathcal{E}_j^N\sigma  \big\Vert_{L^{p(n/(Np))',p}(\rn)}&\lesssim 2^{-j(n/p-n)}2^{jN}\big\Vert \sigma(2^j\cdot)\widehat{\Psi}\big\Vert_{L_N^{\tau^{(N,p)},p}(\rn)}\\
&\lesssim 2^{-j(n/p-n)}2^{jN}\big\Vert \sigma(2^j\cdot)\widehat{\Psi}\big\Vert_{L_s^{\tau^{(s,p)},p}(\rn)}
\end{align*}
by applying dilation, Lemma \ref{hausyoung} with $(p(n/(Np))')'=\tau^{(N,p)}$, and Lemma \ref{embedding} with $s>N$.
Combining with the estimate $\Vert \mathcal{L}_j^{\Theta}a\Vert_{L^1(Q^*)}\lesssim |Q|^{1/2}\Vert \mathcal{L}_j^{\Theta}a\Vert_{L^2(\rn)}$, we finally obtain
\begin{align*}
\mathcal{J}_2&\lesssim \sup_{j\in\zz}{\big\Vert \sigma(2^j\cdot)\widehat{\Psi}\big\Vert_{L_s^{\tau^{(s,p)},p}(\rn)}}|Q|^{1/2}\Big(\sum_{j:2^jl(Q)\ge 1}{2^{-2j(n/p-n)}\Vert \mathcal{L}_j^{\Theta}a\Vert_{L^2(\rn)}^2} \Big)^{1/2}\\
&\lesssim \sup_{j\in\zz}{\big\Vert \sigma(2^j\cdot)\widehat{\Psi}\big\Vert_{L_s^{\tau^{(s,p)},p}(\rn)}}|Q|^{1/p-1/2} \big\Vert \big\{ \mathcal{L}_j^{\Theta}a\big\}_{j\in\zz}\big\Vert_{L^2(\ell^2)}\\
&\lesssim \sup_{j\in\zz}{\big\Vert \sigma(2^j\cdot)\widehat{\Psi}\big\Vert_{L_s^{\tau^{(s,p)},p}(\rn)}}
\end{align*}
because $\big\Vert \big\{ \mathcal{L}_j^{\Theta}a\big\}_{j\in\zz}\big\Vert_{L^2(\ell^2)}\approx \Vert a\Vert_{L^2(\rn)}\le |Q|^{-1/p+1/2} $.

This concludes the proof of the proposition.
\end{proof}

\section{Proof of Theorem \ref{mainnegative}}\label{proofnegative}
The  construction of our counterexamples is based on the idea in \cite{Park} and the following lemma,
 which is  crucial in the proof.
\begin{lemma}\label{keylemma}
Let $0<s,\ga<\infty$ and define the function on $\rn$
\begin{equation}\label{examplekey}
\mathcal{H}^{(s,\gamma)}(x):=\dfrac{1}{(1+4\pi^2|x|^2)^{s/2 }}\dfrac{1}{(1+\ln(1+4\pi^2|x|^2))^{\gamma/2}}.
\end{equation}
Then $\widehat{\mathcal{H}^{(s,\gamma)}}$ is a positive radial function and there exist $c_{s,\ga,n}, d_{s,\ga,n}>0$ such that
\begin{equation}\label{property1}
\widehat{\mathcal{H}^{(s,\gamma)}}(\xi)\le c_{s,\ga,n} e^{-|\xi|/2} \qquad \text{ when }~ |\xi|\ge 1
\end{equation}
and
\begin{equation*}
\dfrac{1}{d_{s,\ga,n}}\le \dfrac{\widehat{\mathcal{H}^{(s,\ga)}}(\xi)}{\mathfrak{T}^{(s,\ga)}(\xi)}\le d_{s,\ga,n} \qquad \text{ when }~ |\xi|\le 1
\end{equation*} where
\begin{equation*}
\mathfrak{T}^{(s,\ga)}(\xi):=\begin{cases}
|\xi|^{-(n-s)}(1+2\ln{|\xi|^{-1}})^{-\gamma/2} & \text{ for } ~ 0<s<n\\
1 & \text{ for } ~ s\ge n.
\end{cases}  
\end{equation*}

\end{lemma}

\begin{proof}
It is known that the Fourier transform of 
 $(1+4\pi^2 |x|^2)^{-s/2}$ is the Bessel potential $G_s(\xi)$. 
Recall that $G_s$ is a positive radial function, $\Vert G_s\Vert_{L^1(\rn)}=1$, and there exist $C_{s,n}, D_{s,n}>0$ such that
\begin{equation}\label{bessel}
G_s(\xi)\le C_{(s,n)} e^{-|\xi|/2} \quad \text{ for }~ |\xi|\ge 1,
\end{equation} and
\begin{equation}\label{bessel2}
\frac{1}{D_{(s,n)}}\le \frac{G_s(\xi)}{\mathfrak{S}_s(\xi)}\le D_{(s,n)}  \quad \text{ for }~ |\xi|\le 1
\end{equation} where
\begin{equation*}
\mathfrak{S}_s(\xi):= \begin{cases}
|\xi|^{-(n-s)} & \text{ for } ~ 0<s<n\\
 \ln{(2|\xi|^{-1})}   & \text{ for } ~ s=n\\
1 & \text{ for } ~ s>n.
\end{cases}
\end{equation*}
 Here we note that for any $\epsilon>0$
 \begin{equation}\label{constantsize}
 C_{(s,n)},D_{(s,n)}\lesssim_{\epsilon,n} e^{\epsilon|s-n|}.
 \end{equation}
We refer to \cite[Ch. 1.2.2]{Gr2} for more details.

Using the identity 
\begin{equation*} 
 A^{-\ga/2} = \frac{1}{\Gamma( \ga/ 2)} \int_0^\infty e^{-t A} t^{ \ga/ 2} \frac{dt}{t} ,
\end{equation*}
 which is valid for $A>0$, we write
\begin{align*}
\big(1+  \log(1+4\pi^2|x|^2) \big)^{-\ga/2}  
& = \frac{1}{\Ga(\ga /2)} 
\int_0^\infty e^{-t} e^{- t  \log(1+4\pi^2|x|^2)} t^{ \ga/ 2} \frac{dt}{t} \\
& = \frac{1}{\Ga( \ga/ 2)} 
\int_0^\infty e^{-t} \frac{1}{ (1+4\pi^2|x|^2)^t } t^{ \ga/ 2} \frac{dt}{t}. 
\end{align*} 
We obtain from this that the Fourier transform  of 
$\big(1+  \log(1+4\pi^2 |x|^2) \big)^{-\ga/2} $ is 
\begin{equation*} 
\frac{1}{\Ga(\ga /2)} 
\int_0^{\infty} e^{-t} G_{2t}(\xi) t^{\ga/ 2} \frac{dt}{t}
\end{equation*} 
and consequently,
\begin{equation*}
\wh{\mathcal{H}^{(s,\gamma)}}(\xi)=G_s \ast \Big( \frac{1}{\Ga(\ga/ 2)} 
\int_0^\infty e^{-t} G_{2t}(\cdot) t^{\ga/ 2} \frac{dt}{t} \Big)(\xi) =\frac{1}{\Ga( \ga/ 2)} \int_0^{\infty} e^{-t} G_{2t+s}(\xi) t^{ \ga/ 2} \frac{dt}{t}.
\end{equation*}
Clearly, $\wh{\mathcal{H}^{(s,\gamma)}}$ is a positive radial function since so is $G_{2t+s}$.

Suppose $|\xi|\ge 1$. Then using (\ref{bessel}) and (\ref{constantsize}) with $0<\epsilon<1/100$,
\begin{align*}
|\widehat{\mathcal{H}^{(s,\gamma)}}(\xi)|&\lesssim_{\epsilon,n} \frac{1}{\Ga(\ga/2)}\Big( \int_0^{\infty}{e^{-t}e^{\epsilon |2t+s-n|}t^{\gamma/2}}\frac{dt}{t}\Big)e^{-|\xi|/2}\lesssim_{s,n,\gamma} e^{-|\xi|/2},
\end{align*}
which proves (\ref{property1}).

Now we assume that $|\xi|\le 1$.
When $0<s<n$
\begin{align*}
\widehat{\mathcal{H}^{(s,\gamma)}}(\xi)=\frac{1}{\Ga( \ga/ 2)} \int_0^{\frac{n-s}{2}} e^{-t} G_{2t+s}(\xi) t^{ \ga/ 2} \frac{dt}{t}+\frac{1}{\Ga( \ga/ 2)} \int_{\frac{n-s}{2}}^{\infty} e^{-t} G_{2t+s}(\xi) t^{ \ga/ 2} \frac{dt}{t}.
\end{align*}
Then using (\ref{bessel2}), (\ref{constantsize}), and change of variables,
\begin{align*}
&\frac{1}{\Ga( \ga/ 2)} \int_0^{\frac{n-s}{2}} e^{-t} G_{2t+s}(\xi) t^{ \ga/ 2} \frac{dt}{t}\\
&\lesssim_{n,\epsilon} |\xi|^{-(n-s)}\frac{1}{\Ga(\ga/2)}\int_{0}^{\frac{n-s}{2}}{e^{-t}|\xi|^{2t}e^{\epsilon(n-2t-s)}t^{\gamma/2}}\frac{dt}{t}\\
 &\le e^{\epsilon(n-s)} |\xi|^{-(n-s)}\frac{1}{\Ga(\ga/2)}\int_0^{\frac{n-s}{2}}{e^{-t(1+2\ln(|\xi|^{-1}))}  t^{\ga/2}      }\frac{dt}{t}\\
 &\le e^{\epsilon(n-s)} |\xi|^{-(n-s)}(1+2\ln(|\xi|^{-1}))^{-\ga/2}\frac{1}{\Ga(\ga/2)}\int_0^{\infty}{e^{-t}t^{\gamma/2}}\frac{dt}{t}\\
 &\lesssim_{s,n,\gamma}|\xi|^{-(n-s)}(1+2\ln(|\xi|^{-1}))^{-\ga/2}
\end{align*}
and
\begin{align*}
&\frac{1}{\Ga( \ga/ 2)} \int_0^{\frac{n-s}{2}} e^{-t} G_{2t+s}(\xi) t^{ \ga/ 2} \frac{dt}{t}\\
&\gtrsim_{n,\epsilon} |\xi|^{-(n-s)}\frac{1}{\Ga(\ga/2)}\int_{0}^{\frac{n-s}{2}}{e^{-t}|\xi|^{2t}e^{-\epsilon(n-2t-s)}t^{\gamma/2}}\frac{dt}{t}\\
    &\ge e^{-\epsilon(n-s)}|\xi|^{-(n-s)}\frac{1}{\Ga(\ga/2)}\int_0^{\frac{n-s}{2}}{e^{-t(1+2\ln(|\xi|^{-1}))}    t^{\ga/2}}\frac{dt}{t}\\
    &\ge e^{-\epsilon(n-s)} |\xi|^{-(n-s)}(1+2\ln(|\xi|^{-1}))^{-\ga/2}\frac{1}{\Ga(\ga/2)}\int_0^{\frac{n-s}{2}}{e^{-t}t^{\gamma/2}}\frac{dt}{t}\\
 &\gtrsim_{s,n,\gamma}|\xi|^{-(n-s)}(1+2\ln(|\xi|^{-1}))^{-\ga/2}.
\end{align*}
Similarly, we can also prove that
\begin{equation*}
\frac{1}{\Ga( \ga/ 2)} \int_{\frac{n-s}{2}}^{\infty} e^{-t} G_{2t+s}(\xi) t^{ \ga/ 2} \frac{dt}{t}\approx_{s,n,\gamma} 1.
\end{equation*}

A similar computation, together with (\ref{bessel2}) and (\ref{constantsize}), will lead to an estimate for $s\ge n$, in which $\widehat{\mathcal{H}^{(s,\gamma)}}\approx_{s,\ga,n} 1$ for $|\xi|\le 1$.
We leave this to the reader to avoid unnecessary repetition.
\end{proof}

In what follows let $\eta,\widetilde{\eta}$ denote Schwartz functions so that $\eta\ge 0$, $\eta(x)\ge c$ on $\{x\in\rn: |x|\le 1/100\}$ for some $c>0$, $\textup{Supp}(\widehat{\eta})\subset \{\xi\in \rn: |\xi|\le 1/1000\}$, $\widehat{\widetilde{\eta}}(\xi)=1$ for $|\xi|\le 1/1000$, and $\textup{Supp}(\widehat{\widetilde{\eta}})\subset \{\xi\in\rn: |\xi|\le 1/100\}$. 
Let $e_1:=(1,0,\dots,0)\in \zn$ and $0<t,\gamma<\infty$. 
Define $\mathcal{H}^{(t,\ga)}$ as in (\ref{examplekey}),
\begin{equation*}
K^{(t,\gamma)}(x):=\mathcal{H}^{(t,\gamma)}\ast \widetilde{\eta}(x)e^{2\pi i\langle x,e_1\rangle},
\end{equation*}
and
\begin{equation*}
\sigma^{(t,\gamma)}(\xi):=\widehat{K^{(t,\gamma)}}(\xi).
\end{equation*}

We investigate an upper bound of $\sup_{j\in\zz}{ \big\Vert \sigma^{(t,\gamma)}(2^j\cdot)\widehat{\Psi}\big\Vert_{L^{r,q}_s(\rn)}}$ and a lower bound of $\Vert T_{\sigma^{(t,\ga)}}\Vert_{H^p(\rn) \to H^p(\rn)}$ when  $t-n<s$.

\subsection{Upper bound of $\sup_{j\in\zz}{ \big\Vert \sigma^{(t,\gamma)}(2^j\cdot)\widehat{\Psi}\big\Vert_{L^{r,q}_s(\rn)}}$}

Note that, due to the supports of $\sigma^{(t,\gamma)}$ and $\widehat{\Psi}$, we have
\begin{equation*}
\sigma^{(t,\gamma)}(2^j\xi)\widehat{\Psi}(\xi)=\begin{cases}
\widehat{K^{(t,\gamma)}}(2^j\xi)\widehat{\Psi}(\xi), & -2\le j\le 2\\
0,& \textup{otherwise.}
\end{cases}
\end{equation*}
For $-2\le j\le 2$ and $t-n<s$,
\begin{align*}
\big\Vert \sigma^{(t,\gamma)}(2^j\cdot)\widehat{\Psi}\big\Vert_{L^{r,q}_s(\rn)}\lesssim \big\Vert \sigma^{(t,\gamma)}\big\Vert_{L^{r,q}_s(\rn)} \lesssim \big\Vert \widehat{\mathcal{H}^{(t,\gamma)}}\big\Vert_{L^{r,q}_s(\rn)}=\big\Vert \widehat{\mathcal{H}^{(t-s,\gamma)}}\big\Vert_{L^{r,q}(\rn)}
\end{align*}
where  Lemma \ref{katoponce} is applied.

For $u>0$ define
\begin{equation*}
\mathcal{T}^{(t-s,\ga)}(u):=\begin{cases}
u^{-(n-t+s)}(1+2\ln u^{-1})^{-\ga/2} & \text{ for }~ u \le 1\\
e^{-u/2+1/2} & \text{ for }~ u>1.
\end{cases}
\end{equation*}
Then $\mathcal{T}^{(t-s,\ga)}$ is a positive decreasing function and this implies that
\begin{equation}\label{invariant}
\big( \mathcal{T}^{(t-s,\ga)}\big)^*(u)= \mathcal{T}^{(t-s,\ga)}(u).
\end{equation}

We first assume $0<q<\infty$.
By using Lemma \ref{keylemma}, we have 
\begin{equation*}
\widehat{\mathcal{H}^{(t-s,\ga)}}(\xi)\lesssim_{s,t,\ga,n}\mathcal{T}^{(t-s,\ga)}(|\xi|),
\end{equation*} from which
\begin{align*}
\big\Vert \widehat{\mathcal{H}^{(t-s,\gamma)}}\big\Vert_{L^{r,q}(\rn)}&\lesssim_{s,t,\ga.n} \big\Vert \mathcal{T}^{(t-s,\gamma)}\big(|\cdot|\big)\big\Vert_{L^{r,q}(\rn)}\\
&=\Big( \int_0^{\infty}{\Big( \mathcal{T}^{(t-s,\ga)}\big( (u/\Omega_n)^{1/n}\big)u^{1/r}\Big)^{q}}\frac{du}{u}\Big)^{1/q}\\
&=\Omega_n^{1/r} n^{1/q}\Big( \int_0^{\infty}{\big( \mathcal{T}^{(t-s,\ga)}(u)\big)^qu^{nq/r}}\frac{du}{u}\Big)^{1/q}
\end{align*}
where Lemma \ref{radialrearrangement} is applied with (\ref{invariant}).
Furthermore,
\begin{align*}
\Big(\int_0^1{\big( \mathcal{T}^{(t-s,\ga)}(u)\big)^q u^{nq/r}}\frac{du}{u} \Big)^{1/q}&=\Big( \int_0^1{\frac{1}{u^{(n-t+s-n/r)q}}\frac{1}{(1+2\ln u^{-1})^{\ga q/2}}}\frac{du}{u}\Big)^{1/q}\\
&=\Big(\int_1^{\infty}{u^{(n-t+s-n/r)q}\frac{1}{(1+2\ln u)^{\ga q/2}}}\frac{du}{u} \Big)^{1/q}
\end{align*}
and
\begin{equation*}
\Big(\int_1^{\infty}{\big( \mathcal{T}^{(t-s,\ga)}(u)\big)^q u^{nq/r}}\frac{du}{u} \Big)^{1/q}=e^{1/2}\Big(\int_1^{\infty}{e^{-uq/2}u^{nq/r}}\frac{du}{u} \Big)^{1/q}\lesssim_{q,r,n} 1
\end{equation*}

Finally, we conclude that
\begin{equation}\label{upper1}
\sup_{j\in\zz}{ \big\Vert \sigma^{(t,\gamma)}(2^j\cdot)\widehat{\Psi}\big\Vert_{L^{r,q}_s(\rn)}} \lesssim_{s,\ga,n,q,r} 1+\Big(\int_1^{\infty}{u^{(n-t+s-n/r)q}\frac{1}{(1+2\ln u)^{\ga q/2}}}\frac{du}{u} \Big)^{1/q}
\end{equation}
and with the usual modification if $q=\infty$ we may also obtain
\begin{equation}\label{upper2}
\sup_{j\in\zz}{ \big\Vert \sigma^{(t,\gamma)}(2^j\cdot)\widehat{\Psi}\big\Vert_{L^{r,\infty}_s(\rn)}} \lesssim_{s,\ga,n,r} 1+\sup_{u>1}{\frac{u^{n-t+s-n/r}}{(1+2\ln u)^{\ga/2}}}.
\end{equation}

\subsection{Lower bound of $\Vert T_{\sigma^{(t,\ga)}}\Vert_{H^p(\rn)\to H^p(\rn)}$}

If $1\le p<\infty$, then
\begin{equation*}
\Vert T_{\sigma^{(t,\gamma)}}\Vert_{H^p(\rn)\to H^p(\rn)}\ge \Vert \sigma^{(t,\gamma)}\Vert_{L^{\infty}(\rn)}\ge |\sigma^{(t,\gamma)}(e_1)|\gtrsim \Vert \mathcal{H}^{(t,\gamma)}\Vert_{L^1(\rn)}.
\end{equation*}
Moreover, for $0<p<1$, define $f(x):=\eta(x)e^{2\pi i\langle x,e_1\rangle}$. Observe that 
\[
\big| T_{\sigma^{(t,\gamma)}}f(x)\big|=\big|\mathcal{H}^{(t,\gamma)}\ast \eta(x)\big| 
\]
 and thus
\begin{align*}
\Vert T_{\sigma^{(t,\gamma)}}\Vert_{H^p(\rn)\to H^p(\rn)}&\gtrsim \Vert T_{\sigma^{(t,\gamma)}}f\Vert_{H^p(\rn)}\ge \Vert T_{\sigma^{(t,\gamma)}}f\Vert_{L^p(\rn)}\\
&=\Vert \mathcal{H}^{(t,\gamma)}\ast \eta\Vert_{L^p(\rn)}\gtrsim \Vert \mathcal{H}^{(t,\gamma)}\Vert_{L^p(\rn)},
\end{align*} 
where the last inequality follows from the fact that $\mathcal{H}^{(t,\gamma)}, \eta\ge 0$ and $\mathcal{H}^{(t,\gamma)}(x-y)\ge \mathcal{H}^{(t,\gamma)}(x) \mathcal{H}^{(t,\gamma)}(y)$.

Consequently, for any $0<p<\infty$,
\begin{align}\label{lower}
&\Vert T_{\sigma^{(t,\gamma)}}\Vert_{H^p(\rn)\to H^p(\rn) }\gtrsim \Vert \mathcal{H}^{(t,\gamma)}\Vert_{L^{\min{(1,p)}}(\rn)}\nonumber\\
&=\Big( \int_{\rn}{\dfrac{1}{(1+4\pi^2|x|^2)^{t\min{(1,p)}/2}}\dfrac{1}{(1+\ln(1+4\pi^2 |x|^2))^{\min{(1,p)}\gamma/2}}}dx\Big)^{1/\min{(1,p)}}.
\end{align}

\subsection{Completion of the proof of Theorem \ref{mainnegative}}
We are only concerned with the case $0<p\le 2$ as the other cases follow by a duality argument.
Suppose $n/p-n/2<s<n/\min{(1,p)}$, $r=\tau^{(s,p)}$, and $\min{(1,p)}<q$. 
Choose 
\begin{equation}\label{range}
2/q<\ga\le 2/\min{(1,p)}
\end{equation} and  let $t=\frac{n}{\min{(1,p)}}$ such that $n-t+s-n/r=0$.
Then
\begin{equation*}
\sup_{j\in\zz}{ \big\Vert \sigma^{(t,\gamma)}(2^j\cdot)\widehat{\Psi}\big\Vert_{L^{r,q}_s(\rn)}} \lesssim_{s,\ga,n,q} 1+\Big(\int_1^{\infty}{\frac{1}{(1+2\ln u)^{\ga q/2}}}\frac{du}{u} \Big)^{1/q}\lesssim 1
\end{equation*} because of (\ref{range}) for $0<q<\infty$, and similarly, $\sup_{j\in\zz}{ \big\Vert \sigma^{(t,\gamma)}(2^j\cdot)\widehat{\Psi}\big\Vert_{L^{r,\infty}_s(\rn)}}\lesssim_{s,\ga,n} 1$ for $q=\infty$.
On the other hand, $\Vert T_{\sigma^{(t,\gamma)}}\Vert_{H^p(\rn)\to H^p(\rn) }$ is bounded below by
\begin{equation*}
 \Big( \int_{\rn}{\dfrac{1}{(1+4\pi^2 |x|^2)^{n/2}}\dfrac{1}{(1+\ln(1+4\pi^2 |x|^2))^{\min{(1,p)}\ga/2}}}dx\Big)^{1/\min{(1,p)}},
\end{equation*} which diverges for the choice of $\ga$ in (\ref{range}).

\appendix
\section{Complex Interpolation of $H^1$- and $L^2$-boundedness}

In this section, we review the complex interpolation method of Calder\'on-Torchinsky \cite{Ca_To} and Triebel \cite{Tr}, which is a generalization of the well-known method of Calder\'on \cite{Ca} and Fefferman and Stein \cite{Fe_St2}.

Let $A:=\{z\in \cc : 0<\textup{Re}(z)<1 \}$ be a strip in the complex plane $\cc$ and $\overline{A}$ denote its closure.
We say that the mapping $z \mapsto f_z\in \mathscr S'(\rn)$ is a $\mathscr S'$-analytic function on $A$ if the following properties are satisfied:
\begin{enumerate}
\item For any $\varphi\in \mathscr S(\rn)$ with compact support,
$g(x,z):=\big(\varphi \widehat{f_z} \big)(x)$ is a uniformly continuous and bounded function on $\rn\times \overline{A}$.
\item For any $\varphi\in \mathscr S(\rn)$ with compact support and any fixed $x\in\rn$,
$h_x:=\big( \varphi \widehat{f_z}\big)^{\vee} $ is an analytic function on $A$.
\end{enumerate}

Let $0<p_0,p_1<\infty$.
Then we define
$F\big(H^{p_0}(\rn),H^{p_1}(\rn)\big)$ to be the collection of all $\mathscr S'$-analytic functions $f_z$ on $A$ such that 
\begin{equation*}
f_{it}\in H^{p_0}(\rn), \qquad f_{1+it}\in H^{p_1}(\rn) \quad  \text{ for any }~ t\in \rr
\end{equation*} and
\begin{equation*}
\sup_{t\in \rr}{\Vert f_{l+it}\Vert_{H^{p_l}(\rn)}}<\infty \qquad \text{ for each }~ l=1,2.
\end{equation*}
Moreover, 
\begin{equation*}
\Vert f_z\Vert_{F(H^{p_0}(\rn),H^{p_1}(\rn))}:=\max{\Big( \sup_{t\in \rr}{\Vert f_{it}\Vert_{H^{p_0}(\rn)}}, \sup_{t\in \rr}{\Vert f_{1+it}\Vert_{H^{p_1}(\rn)}}\Big)}.
\end{equation*}
For $0<\theta<1$ the intermediate space $( H^{p_0}(\rn), H^{p_1}(\rn))_{\theta}$ is defined by
\begin{equation*}
\big( H^{p_0}(\rn), H^{p_1}(\rn)\big)_{\theta}:=\big\{g: \exists f_z\in F\big(H^{p_0}(\rn),H^{p_1}(\rn)\big) \text{ so that } g=f_{\theta} \big\}
\end{equation*}
and the (quasi-)norm in the space is 
\begin{equation*}
\Vert g\Vert_{( H^{p_0}(\rn), H^{p_1}(\rn))_{\theta}}:=\inf_{f_z\in F(H^{p_0}(\rn),H^{p_1}(\rn)):g=f_{\theta}}{\Vert f_z\Vert_{F(H^{p_0}(\rn),H^{p_1}(\rn))}}
\end{equation*}
where the infimum is taken over all admissible functions $f_z$ in the sense that $f_z\in F\big(H^{p_0}(\rn),H^{p_1}(\rn)\big)$ and $g=f_{\theta}$.
It is known in \cite{Ca_To,Tr} that for any $0<p_0,p_1<\infty$ and $0<\theta<1$
\begin{equation}\label{hardyinterpol}
\big(H^{p_0}(\rn), H^{p_1}(\rn)\big)_{\theta}= H^p(\rn) \quad \text{ when }~1/p=(1-\theta)/p_0+\theta/p_1.
\end{equation} 

We now use this method to interpolate $H^1$- and $L^2$-boundedness of the multiplier operator $T_{\sigma}$ to obtain $L^p$ estimates for $1<p<2$. Note that $H^p(\rn)=L^p(\rn)$ for $1<p<\infty$.
Since most arguments are very similar to that used in the proof  of \cite[Theorem 3.1]{Gr_Sl}, we shall provide only the outline of the proof, omitting the details.

We may consider a Schwartz function $f$ whose Fourier transform is compactly supported via a  density argument.
Suppose that $1<p<2$ and $n/p-n/2<s<n$. Let $0<\theta<1$ satisfy $1/p=(1-\theta)/1+\theta/2$. 
Then we have $s>n/p-n/2=(1-\theta)n/2$. Pick $s_0>n/2$ so that
\begin{equation*}
s>(1-\theta)s_0>(1-\theta)n/2
\end{equation*} and let $s_1:=\frac{s-(1-\theta)s_0}{\theta}>0$ which implies
\begin{equation*}
s=(1-\theta)s_0+\theta s_1.
\end{equation*}

Since $f\in L^p(\rn)=H^p(\rn)=\big(H^1(\rn),H^2(\rn)\big)_{\theta}$, by definition, for any $\epsilon>0$, there exists $f_z^{\epsilon}\in F\big(H^1(\rn),H^2(\rn)\big)$ such that $f=f_{\theta}^{\epsilon}$ and 
\begin{equation}\label{intermediate}
\Vert f_z^{\epsilon} \Vert_{F(H^1(\rn),H^2(\rn))}<\Vert f\Vert_{(H^1(\rn),H^2(\rn))_{\theta}}+\epsilon.
\end{equation}

Now let $\widehat{\Theta}(\xi):=\widehat{\Psi}(\xi/2)+\widehat{\Psi}(\xi)+\widehat{\Psi}(2\xi)$ as before, and $\sigma^{j,s}:=(I-\Delta)^{s/2}\big( \sigma(2^j\cdot)\widehat{\Psi}\big) $ for each $j\in\zz$. 
We define, as in \cite[(3.18)]{Gr_Sl},
\begin{equation*}
\sigma_z(\xi):=\frac{(1+\theta)^{n+1}}{(1+z)^{n+1}}\sum_{j\in\zz}{(I-\Delta)^{-\frac{s_0(1-z)+s_1z}{2}}\Big(\sigma^{j,s} h_{j,s}^{\frac{s-(1-z)s_0-zs_1}{n}}\Big)(\xi/2^j)\widehat{\Theta}(\xi/2^j)}
\end{equation*}
where $h_{j,s}:\rn \to (0,\infty)$ is a measure preserving transformation so that $|\sigma^{j,s}|=(\sigma^{j,s})^*\circ h_{j,s}$. Then we note that $\sigma_{\theta}=\sigma$ and $F_z:=T_{\sigma_{z}}f_z^{\epsilon}$ is a 
$\mathscr S'$-analytic function on $A$.
Moreover, 
\begin{align*}
\Vert T_{\sigma}f\Vert_{H^p(\rn)}&\approx \big\Vert T_{\sigma_{\theta}}f_{\theta}^{\epsilon}\big\Vert_{(H^1(\rn),H^2(\rn))_{\theta}}=\Vert F_{\theta}\Vert_{(H^1(\rn),H^2(\rn))_{\theta}}\\
&\le \Vert F_z\Vert_{F(H^1(\rn),H^2(\rn))}=\max{\Big(\sup_{t\in\rr }{ \Vert F_{it}\Vert_{H^1(\rn)}},\sup_{t\in\rr}{\Vert F_{1+it}\Vert_{H^2(\rn)}}    \Big)}.
\end{align*}
By using Theorem \ref{mainpositive} for $p=1$, we have
\begin{align*}
\Vert F_{it}\Vert_{H^1(\rn)}&=\Vert T_{\sigma_{it}}f_{it}^{\epsilon}\Vert_{H^1(\rn)}\lesssim \sup_{j\in \zz}{\big\Vert \sigma_{it}(2^j\cdot)\widehat{\Psi}\big\Vert_{L_{s_0}^{n/s_0,1}(\rn)}}\Vert f_{it}^{\epsilon}\Vert_{H^1(\rn)}\\
&\lesssim \sup_{j\in \zz}{\big\Vert \sigma_{it}(2^j\cdot)\widehat{\Psi}\big\Vert_{L_{s_0}^{n/s_0,1}(\rn)}}\Big( \Vert f\Vert_{(H^1(\rn),H^2(\rn))_{\theta}} +\epsilon\Big),
\end{align*} where (\ref{intermediate}) is applied in the last inequality. Similarly, with $L^2$-boundedness,
\begin{align*}
\Vert F_{1+it}\Vert_{H^2(\rn)}&=\Vert T_{\sigma_{1+it}}f_{1+it}^{\epsilon}\Vert_{H^2(\rn)}\lesssim \Vert \sigma_{1+it}\Vert_{L^{\infty}(\rn)}\Vert f_{1+it}^{\epsilon}\Vert_{H^2(\rn)}\\
&\lesssim \sup_{j\in \zz}{\big\Vert \sigma_{1+it}(2^j\cdot)\widehat{\Psi}\big\Vert_{L^{\infty}(\rn)}}\Big( \Vert f\Vert_{(H^1(\rn),H^2(\rn))_{\theta}} +\epsilon\Big).
\end{align*} 
Therefore, once we prove
\begin{equation}\label{mainappendix}
{\big\Vert \sigma_{it}(2^j\cdot)\widehat{\Psi}\big\Vert_{L_{s_0}^{n/s_0,1}(\rn)}}, {\big\Vert \sigma_{1+it}(2^j\cdot)\widehat{\Psi}\big\Vert_{L^{\infty}(\rn)}}\lesssim \big\Vert \sigma(2^j\cdot)\widehat{\Psi}\big\Vert_{L_s^{n/s,1}(\rn)}
\end{equation} uniformly in $j$,
then we are done by using (\ref{hardyinterpol}) and taking $\epsilon\to 0$.

Let us prove (\ref{mainappendix}).
We first observe that
\begin{align*}
&\sigma_z(2^j\xi)\widehat{\Psi}(\xi)\\
&=\frac{(1+\theta)^{n+1}}{(1+z)^{n+1}}\sum_{k\in\zz}{(I-\Delta)^{-\frac{s_0(1-z)+s_1z}{2}}\Big(\sigma^{k,s} h_{k,s}^{\frac{s-(1-z)s_0-zs_1}{n}}\Big)(\xi/2^{k-j})\widehat{\Theta}(\xi/2^{k-j})}\widehat{\Psi}(\xi)
\end{align*} is actually finite sum over $k$ near $j$ due to the supports of $\widehat{\Theta}$ and $\widehat{\Psi}$, and for simplicity, we may therefore take $k=j$ in the calculation below.

Using Lemma \ref{katoponce}, we have
\begin{equation*}
{\big\Vert \sigma_{it}(2^j\cdot)\widehat{\Psi}\big\Vert_{L_{s_0}^{n/s_0,1}(\rn)}}\lesssim \frac{1}{(1+|t|^2)^{(n+1)/2}}\Big\Vert (I-\Delta)^{\frac{(s_0-s_1)it}{2}}\Big(\sigma^{j,s}h_{j,s}^{\frac{s-s_0+(s_0-s_1)it}{n}}\Big)\Big\Vert_{L^{n/s_0,1}(\rn)}.
\end{equation*}
Then we apply \cite[Lemma 3.5, 3.7]{Gr_Sl} to bound this by
\begin{align*}
&\Big\Vert \sigma^{j,s}h_{j,s}^{\frac{s-s_0+(s_0-s_1)it}{n}}\Big\Vert_{L^{n/s_0,1}(\rn)}\lesssim \big\Vert (\sigma^{j,s})^*(r)r^{(s-s_0)/n}\big\Vert_{L^{n/s_0,1}(0,\infty)}\\
&\lesssim \Vert (\sigma^{j,s})^*\Vert_{L^{n/s,1}(0,\infty)}\lesssim \Vert \sigma^{j,s}\Vert_{L^{n/s,1}(\rn)}=\big\Vert \sigma(2^j\cdot)\widehat{\Psi}\big\Vert_{L_s^{n/s,1}(\rn)}.
\end{align*}

On the other hand, using \cite[Lemma 3.4, 3.5, 3.7]{Gr_Sl},
\begin{align*}
&{\big\Vert \sigma_{1+it}(2^j\cdot)\widehat{\Psi}\big\Vert_{L^{\infty}(\rn)}}\\
&\lesssim \frac{1}{(1+|t|^2)^{(n+1)/2}}\Big\Vert (I-\Delta)^{-s_1/2}(I-\Delta)^{(s_0-s_1)it/2}\Big(\sigma^{j,s}h_{j,s}^{\frac{s-s_1+(s_0-s_1)it}{n}} \Big)\Big\Vert_{L^{\infty}(\rn)}\\
 &\lesssim \frac{1}{(1+|t|^2)^{(n+1)/2}}\Big\Vert (I-\Delta)^{(s_0-s_1)it/2}\Big(\sigma^{j,s}h_{j,s}^{\frac{s-s_1+(s_0-s_1)it}{n}} \Big)\Big\Vert_{L^{n/s_1,1}(\rn)}\\
 &\lesssim \Big\Vert \sigma^{j,s}h_{j,s}^{\frac{s-s_1+(s_0-s_1)it}{n}} \Big\Vert_{L^{n/s_1,1}(\rn)} \lesssim \big\Vert (\sigma^{j,s})^*(r)r^{(s-s_1)/n}\big\Vert_{L^{n/s_1,1}(0,\infty)}\\
&\lesssim \Vert (\sigma^{j,s})^*\Vert_{L^{n/s,1}(0,\infty)}\lesssim \Vert \sigma^{j,s}\Vert_{L^{n/s,1}(\rn)}=\big\Vert \sigma(2^j\cdot)\widehat{\Psi}\big\Vert_{L_s^{n/s,1}(\rn)},
\end{align*}
which finishes the proof of (\ref{mainappendix}).

\medskip
\noindent {\bf Acknowledgment:} We would like to thank Professors M. Mastylo and A. Seeger for providing us important references related to real interpolation. We would also like to thank A. Seeger for pointing out to us the 
content of the remark after Theorem \ref{mainpositive}. We are grateful to the anonymous referees for a careful reading and useful comments.


\begin{thebibliography}{99}

\bibitem{Ba_Sa}
A. Baernstein II and E. T. Sawyer,  \emph{Embedding and multiplier theorems for $H^p(\mathbb{R}^n)$}, Mem. Amer. Math. Soc. \textbf{318}
(1985).


\bibitem{Be_Sh}
C. Bennett and R. Sharpley, \emph{Interpolation of operators}, Academic Press, Boston, 1988.



\bibitem{Be_Lo}
J. Bergh and J. L\"ofstr\"om, \emph{Interpolation Spaces, An Introduction}, Springer-Verlag, New York, 1976.


\bibitem{Ca}
A.P. Calder\'on, \emph{Intermediate spaces and interpolation, the complex method}, Studia Math. \textbf{24} (1964) 113-190.




\bibitem{Ca_To}
A.P. Calder\'on and A. Torchinsky, \emph{Parabolic maximal functions associated with a distribution, II},  Adv. Math. \textbf{24} (1977) 101-171.

\bibitem{Cw}
M. Cwikel, \emph{On $(L^{p_0}(A_0),L^{p_1}(A_1))_{\theta,q}$},  Proc. Amer. Math. Soc. \textbf{44} (1974) 286-292.

\bibitem{Fe_Ri_Sa}
C. Fefferman, N. Riviere, and Y. Sagher,  \emph{Interpolation between $H^p$ spaces: The real method}, Trans. Amer. Math. Soc. \textbf{191} (1974) 75-81.




\bibitem{Fe_St2}
C. Fefferman and E. M. Stein,  \emph{$H^p$ spaces of several variables}, Acta Math. \textbf{129} (1972) 137-193.


\bibitem{Gr2}
L. Grafakos,  \emph{Modern Fourier Analysis}, 3rd edition, Graduate Texts in Mathematics 250, Springer, NY 2014.



\bibitem{Gr_Sl}
L. Grafakos and L. Slav\'ikov\'a,  \emph{A sharp version of the H\"ormander multiplier theorem}, Int. Math. Research Notices \textbf{15} (2019) 4764-4783.

\bibitem{Gr_He_Ho_Ng}
L. Grafakos, D. He, P. Honz\'ik, and H. V. Nguyen, \emph{The H\"ormander multiplier theorem I : The linear case revisited}, Illinois J. Math. \textbf{61} (2017) 25-35.


\bibitem{Hol}
T. Holmstedt,  \emph{Interpolation of quasi-normed spaces}, Math. Scand.  \textbf{26}
(1970) 177-199.

\bibitem{Ho}
L. H\"ormander,  \emph{Estimates for translation invariant operators in $L_p$ spaces}, Acta Math.  \textbf{104}
(1960) 93-140.

\bibitem{Ka_Po}
T. Kato and G. Ponce, \emph{Commutator estimates and the Euler and Navier-Stokes equations}, Comm. Pure Appl. Math. \textbf{41} (1988) 891-907.


\bibitem{Mik}
S. G. Mihlin,  \emph{On the multipliers of Fourier integrals}, Dokl. Akad. Nauk SSSR (N.S.)  \textbf{109}
(1956) 701-703 (Russian).


\bibitem{Park}
B. Park,  \emph{Fourier multiplier theorems for Triebel-Lizorkin spaces}, Math. Z. \textbf{293} (2019) 221-258.


\bibitem{Se1}
A. Seeger,  \emph{A limit case of the H\"ormander multiplier theorem}, Monatsh. Math. \textbf{105} (1988) 151-160.

\bibitem{Se2}
A. Seeger,  \emph{Estimates near $L^1$ for Fourier multipliers and maximal functions}, Arch. Math. (Basel) \textbf{53} (1989) 188-193.

\bibitem{Se3}
A. Seeger,  \emph{Remarks on singular convolution operators}, Studia Math. \textbf{97} (1990) 91-114.

\bibitem{Se_Tr}
A. Seeger and W. Trebels,  \emph{Embeddings for spaces of Lorentz-Sobolev type}, Math. Ann. \textbf{373} (2019) 1017-1056.



\bibitem{Sl}
L. Slav\'ikov\'a,  \emph{On the failure of the H\"ormander multiplier theorem in a limiting case}, Rev. Mat. Iber. 
\textbf{36}  (2020)   1013--1020. 

\bibitem{Ta_We}
M. Taibleson and G. Weiss \emph{The molecular characterization of certain Hardy spaces}, Ast\'erisque \textbf{77} (1980) 67-151.

\bibitem{Tr}
H. Triebel, \emph{Complex interpolation and Fourier multipliers for the spaces $B_{p,q}^{s}$ and $F_{p,q}^{s}$ of Besov-Hardy-Sobolev type : The case $0<p\le\infty$, $0<p\le \infty$}, Math. Z. \textbf{176}
(1981) 495-510.

\bibitem{Tr1}
H. Triebel, \emph{Theory of Function Spaces}, Birkhauser, Basel-Boston-Stuttgart
(1983).


\end{thebibliography}
\end{document}